\documentclass[letter, 11pt, reqno]{amsart}
\usepackage[includeheadfoot, margin=1in]{geometry}
\usepackage{amsfonts}
\usepackage{amsmath}
\usepackage{amsthm}
\usepackage{color}
\usepackage[mathscr]{euscript}
\usepackage{bm}
\usepackage{amscd}
\usepackage{comment}
\usepackage[numbers,square]{natbib}
\usepackage{hyperref}
\usepackage{graphicx}
\usepackage{epstopdf}

\newtheorem{Thm}{Theorem}
\newtheorem*{Lem}{Lemma}

\theoremstyle{remark}
\newtheorem*{Remark}{Remark}
\newtheorem*{Ack}{Acknowledgements}

\begin{document}

\title{$G$-Corks \& Heegaard Floer Homology}
\author{Biji Wong}
%\date{\today}
\address{Department of Mathematics, Brandeis University, MS 050, 415 South Street, Waltham, MA 02453}
\email{wongb@brandeis.edu}
\urladdr{https://sites.google.com/a/brandeis.edu/bijiwong}
\maketitle

\begin{abstract}
In \cite{AKMR16}, Auckly-Kim-Melvin-Ruberman showed that for any finite subgroup $G$ of $SO(4)$ there exists a contractible smooth 4-manifold with an effective $G$-action on its boundary so that the twists associated to the non-trivial elements of $G$ don't extend to diffeomorphisms of the entire manifold. We give a different proof of this phenomenon using the Heegaard Floer theoretic argument in \cite{AkKara11}.
\end{abstract}

\section{Introduction}

A cork is a contractible smooth 4-manifold with an involution on its boundary that does not extend to a diffeomorphism of the entire manifold. The first example of a cork was given by Akbulut in \cite{Ak91}. Since then, other examples have been constructed by Akbulut-Yasui in \cite{AkYa08}, Akbulut-Yasui in \cite{AkYa09}, and Tange in \cite{Tang16, Tange162}. Corks can be used to detect exotic structures, see \cite{Ak91, AkYa08, AkYa09, AKMR16, BiGo96, Tang16, Tange162}. In fact, any two smooth structures on a closed simply-connected 4-manifold differ by a single cork twist, see \cite{CFHS96, M96}. A cork twist removes an embedded cork and reglues it using the involution. The involution on the boundary of a cork can be regarded as a $\mathbb{Z}_2$-action, so it is natural to ask if contractible smooth 4-manifolds with other kinds of effective group actions on the boundaries can also be used to detect exotic structures. A number of recent papers have answered this in the affirmative, constructing examples of $G$-corks, $G \neq \mathbb{Z}_2$, that embed inside closed smooth 4-manifolds so that removing and regluing using the $|G|$ twists produces $|G|$ distinct smooth structures, see \cite{AKMR16, GompfCorks1, GompfCorks2, AkCorks, Tange162}. A $G$-cork is a contractible smooth 4-manifold with an effective $G$-action on its boundary so that the twists associated to the non-trivial elements of $G$ do not extend to diffeomorphisms of the entire manifold. 

The purpose of this paper is to use the Heegaard Floer theoretic argument in \cite{AkKara11} to give a different proof that the examples in \cite{AKMR16} are in fact $G$-corks. These examples are defined as follows. Fix a finite subgroup $G$ of $SO(4)$. Let $n = |G|$. Let $\mathbb{W}$ be the Akbulut cork from \cite{Ak91} shown in Figure \ref{fig:1}.

\begin{figure}[h]
\centering
\def\svgwidth{200pt}
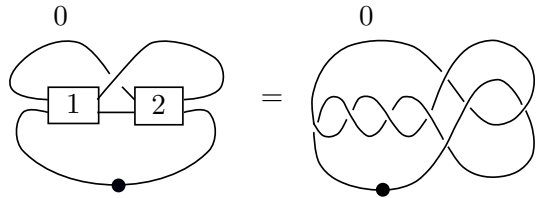
\caption{The Akbulut cork}
\label{fig:1}
\end{figure}

There is an isotopy of $S^3$ that interchanges the two link components on the right side of Figure \ref{fig:1}. This gives an involution $\tau$ of $ \partial \mathbb{W}$. Let $\mathbb{S}$ denote the boundary sum $\mathrel{\natural}_n \mathbb{W}$. Note that its boundary $\partial \mathbb{S}$ inherits an involution $\sigma$. Define $\overline{\mathbb{T}}$ to be the boundary sum of $B^4$ with $n$ copies of $\mathbb{S}$, namely $\overline{\mathbb{T}} = B^4 \mathrel{\natural} \big( G \times \mathbb{S} \big)$. So that we get a well-defined action of $G$ on $\partial \overline{\mathbb{T}}$, we assume that the $n$ copies of $\mathbb{S}$ are attached along 3-balls in $\partial B^4$ that form a principal orbit under the linear action of $G$ on $\partial B^4$. Then we define the action of $G$ on $\partial \overline{\mathbb{T}}$ to be the linear action on $\partial B^4$ and left multiplication on the copies of $\partial \mathbb{S}$. To get the $G$-cork, we twist a copy $\mathbb{S}'$ of $\mathbb{S}$ in the interior of $1 \times \mathbb{S} \subset \overline{\mathbb{T}}$ by the involution $\sigma$. Let $\mathbb{T}$ denote the resulting 4-manifold $B^4 \mathrel{\natural} \Big(\mathbb{S}' \cup_{\sigma} \big( (1 \times \mathbb{S}) - \mathbb{S}' \big) \Big) \mathrel{\natural} \big( (G -1) \times \mathbb{S} \big) $. Note that the action of $G$ on $\partial \overline{\mathbb{T}}$ descends to an action of $G$ on $\partial \mathbb{T}$. Then we have the following:

\begin{Thm}\label{Thm1}
$\mathbb{T}$ is a G-cork.
\end{Thm}

\noindent The Heegaard Floer theoretic argument also yields the following easy consequence:

\begin{Thm}\label{Thm2}
The $G$-action on $\partial \mathbb{T}$ induces an effective $G$-action on $HF^{+}(-\partial \mathbb{T}, \mathfrak{s})$, where $HF^{+}$ is the plus version of Heegaard Floer homology and $\mathfrak{s}$ is the unique Spin$^{c}$ structure on $\partial \mathbb{T}$.
\end{Thm}

We assume the reader is familiar with the basics of Heegaard Floer homology for 3 and 4-manifolds, contact geometry, Stein structures, and Lefschetz fibrations. We use $\mathbb{Z}_2$ coefficients throughout to avoid ambiguity in sign.

\begin{Ack} 
The author would like to thank Danny Ruberman for helpful conversations. A part of this work was supported by an NSF IGERT fellowship under grant number DGE-1068620.
\end{Ack}

\section{Proofs of Theorems \ref{Thm1} \& \ref{Thm2}}

We prove Theorem \ref{Thm1} first. We start by equipping $\mathbb{T}$ with the handle decomposition in Figure \ref{fig:2}. 

\begin{figure}[h]
\centering
\def\svgwidth{210pt}
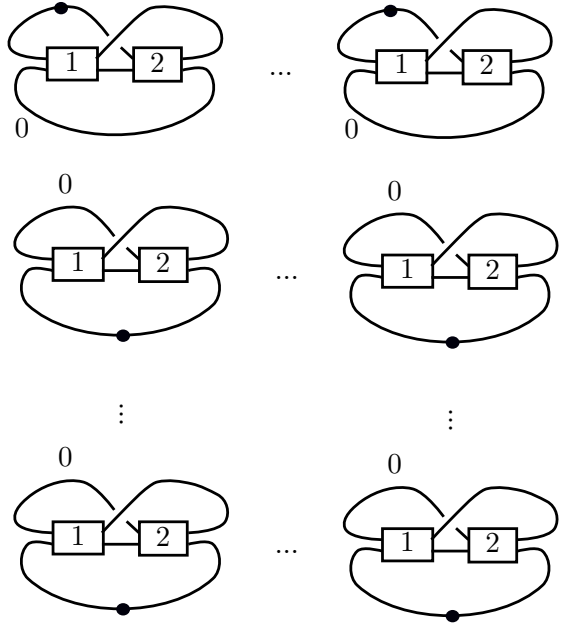
\caption{Handle decomposition of $\mathbb{T}$: There are $n$ rows. The first row represents $\mathbb{S}' \cup_{\sigma} \big( (1 \times \mathbb{S}) - \mathbb{S}' \big)$. Each row after represents a copy of $\mathbb{S}$.}
\label{fig:2}
\end{figure}

\noindent By \cite[Lemma 5.3]{AkYa08}, $\mathbb{T}$ can be given a Stein structure that extends the standard Stein structure on $B^4$. Then $\partial \mathbb{T}$ inherits a contact structure $\xi$. Now fix $g \in G$, $g \neq 1$; by abuse of notation we view this as a diffeomorphism of $\partial \mathbb{T}$. We want to show that $g$ does not extend to a diffeomorphism of $\mathbb{T}$. Let $\mathfrak{s}$ denote the unique Spin$^{c}$ structure on $\mathbb{T}$; we also write $\mathfrak{s}$ for its restriction to $\partial \mathbb{T}$. By puncturing $\mathbb{T}$ in the interior we can view $\mathbb{T}$ as a cobordism from $-\partial \mathbb{T}$ to $S^3$. Then we get the cobordism map $F^{+}_{\mathbb{T}, \mathfrak{s}}: HF^{+}(-\partial \mathbb{T}, \mathfrak{s}) \rightarrow HF^{+}(S^3)$, see \cite{OZ} for details. The twist $g$ induces a second cobordism map $g^{*}: HF^{+}(-\partial \mathbb{T}, \mathfrak{s}) \rightarrow HF^{+}(-\partial \mathbb{T}, \mathfrak{s})$ via the cobordism $(-\partial \mathbb{T} \times [0,\frac{1}{2}]) \mathrel{\cup_{g}} (-\partial \mathbb{T} \times [\frac{1}{2},1])$. The bulk of the proof lies in showing the following: 

\begin{Lem}
Let $c^{+}(\xi) \in HF^{+}(-\partial \mathbb{T}, \mathfrak{s})$ denote the contact invariant associated to $\xi$. Then $F^{+}_{\mathbb{T}, \mathfrak{s}} \big( c^{+}(\xi) \big) \neq 0$, but $F^{+}_{\mathbb{T}, \mathfrak{s}} \circ g^{*} \big( c^{+}(\xi) \big) =0$.
\end{Lem}

\begin{proof}
First attach a 2-handle to $\partial \mathbb{T}$ along a trefoil with framing 1 as in Figure \ref{fig:3}. 

\begin{figure}[h] 
\centering
\def\svgwidth{240pt}
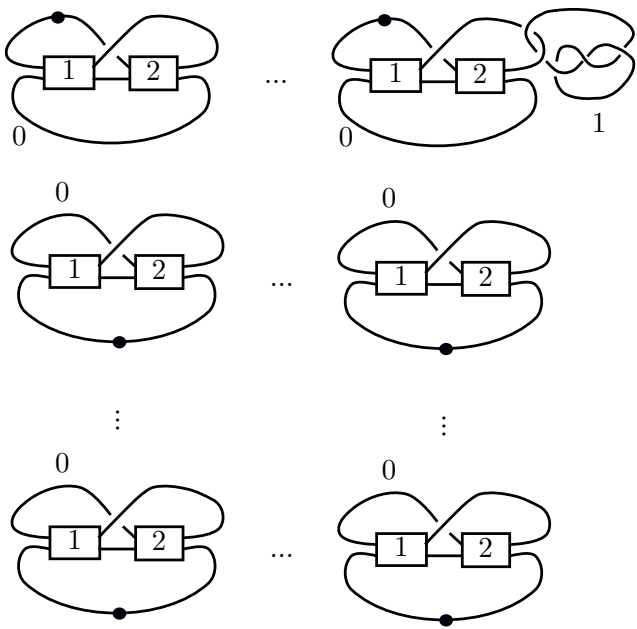
\caption{$\mathbb{T} \mathrel{\cup} $ 2-handle}
\label{fig:3}
\end{figure}

\noindent If we replace the dotted 1-handle linking the trefoil with a pair of 3-balls and put the trefoil in Legendrian position, then we see that the Thurston-Bennequiun number of the tangle is 2, which is 1 more than the framing we started with. By Eliashberg's criterion \cite[Proposition 2.3]{GompfHandlebodiesStein}, the Stein structure on $\mathbb{T}$ can be extended over the 2-handle. Let $M$ denote the cobordism on $\partial \mathbb{T}$ induced by the 2-handle attachment. Then $M$ inherits a Stein structure. By \cite[Lemma 3.6]{AkKara11}, we can extend $M$ to a concave symplectic filling $V$ of $(\partial \mathbb{T}, \xi)$ so that the closed smooth 4-manifold $X := \mathbb{T} \cup V$ has $b^{+}_2 > 1$ and admits the structure of a relatively minimal Lefschetz fibration over $S^2$ with generic fiber of genus $> 1$. Furthermore, if $\mathfrak{t}$ denotes the canonical Spin$^{c}$ structure on $X$ (and also its restriction to $V$), then by \cite[Theorem 3.2 \& Lemma 3.6]{AkKara11},
\[
F^{\text{mix}}_{X,\mathfrak{t}}(\theta^{-}_{(-2)}) = \theta^{+}_{(0)},
\]
\[
F^{\text{mix}}_{V,\mathfrak{t}}(\theta^{-}_{(-2)}) = c^{+}(\xi).
\]
Here $F^{\text{mix}}_{X,\mathfrak{t}}$ is the mixed homomorphism $HF^{-}(S^3) \rightarrow HF^{+}(S^3)$ obtained by puncturing $X$ twice, with one puncture in the interior of $V$ and the other puncture equal to the above puncture of $\mathbb{T}$, $F^{\text{mix}}_{V,\mathfrak{t}}$ is the mixed homomorphism $HF^{-}(S^3) \rightarrow HF^{+}(-\partial \mathbb{T}, \mathfrak{s})$ obtained by puncturing $V$ in the same location as above, $\theta^{-}_{(-2)}$ is the generator of $HF^{-}(S^3)$ with absolute grading --2, and $\theta^{+}_{(0)}$ is the generator of $HF^{+}(S^3)$ with absolute grading 0, see \cite{OZ} for details. Putting this together, we get
\[
0 \neq \theta^{+}_{(0)} = F^{\text{mix}}_{X, \mathfrak{t}}(\theta^{-}_{(-2)}) = F^{+}_{\mathbb{T}, \mathfrak{s}} \circ F^{\text{mix}}_{V,\mathfrak{t}}(\theta^{-}_{(-2)}) = F^{+}_{\mathbb{T}, \mathfrak{s}} \big( c^{+}(\xi) \big). 
\]
\indent All that remains to show is that $F^{+}_{\mathbb{T}, \mathfrak{s}} \circ g^{*} \big( c^{+}(\xi) \big) =0$. Let $X'$ denote $\mathbb{T} \cup_{g} V$ obtained from $X=\mathbb{T} \cup V$ by removing $\mathbb{T}$ and regluing it with the diffeomorphism $g: \partial \mathbb{T} \rightarrow \partial \mathbb{T}$. In $X'$ we have $\mathbb{T} \cup_{g} M$ which admits the following handle decomposition:

\begin{figure}[h] 
\centering
\def\svgwidth{240pt}
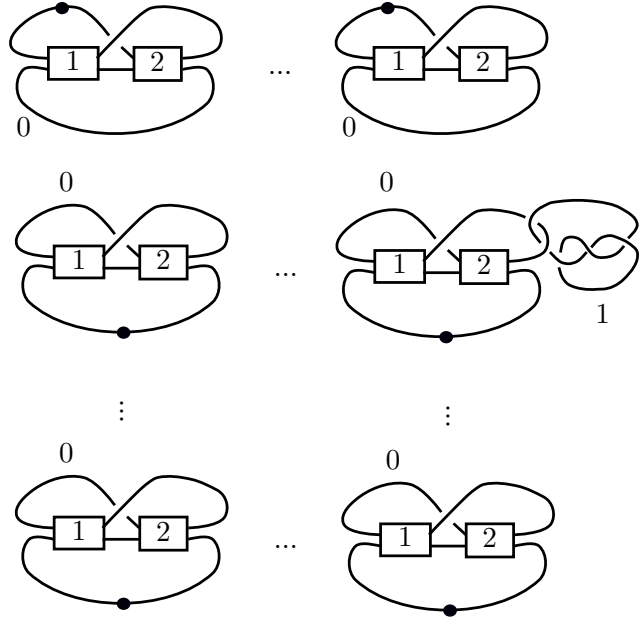
\caption{$\mathbb{T} \cup_{g} M$, where the second row represents $(\{g\} \times \mathbb{S}) \mathrel{\cup}$ 2-handle}
\label{fig:4}
\end{figure}

\noindent Note that the trefoil, thought of in $X'$, gives rise to an embedded torus of self-intersection 1. By \cite[Theorem 3.1]{AkKara11}, $X'$ does not have any basic classes, so for every Spin$^{c}$ structure $\mathfrak{t}'$ on $X'$, the mixed homomorphism $F^{\text{mix}}_{X',\mathfrak{t}'} : HF^{-}(S^3) \rightarrow HF^{+}(S^3)$ is identically zero. Separately, we have a homeomorphism $X \rightarrow X'$ that is the identity on $V$. Let $\mathfrak{t'}$ denote the Spin$^{c}$ structure on $X'$ that corresponds to the canonical Spin$^{c}$ structure $\mathfrak{t}$ on $X$. Note that $\mathfrak{t'}$ restricted to $V$ is the Spin$^{c}$ structure $\mathfrak{t}$. Then we have
\[
0 = F^{\text{mix}}_{X', \mathfrak{t'}}(\theta^{-}_{(-2)}) = F^{+}_{\mathbb{T}, \mathfrak{s}} \circ g^{*} \circ F^{\text{mix}}_{V,\mathfrak{t}}(\theta^{-}_{(-2)}) = F^{+}_{\mathbb{T}, \mathfrak{s}} \circ g^{*} \big( c^{+}(\xi) \big).
\]
\end{proof}

Having proved the main technical lemma, we now finish off the proof of Theorem \ref{Thm1}. Suppose to the contrary the diffeomorphism $g: \partial \mathbb{T} \rightarrow \partial \mathbb{T}$ extends to a diffeomorphism $\widetilde{g}: \mathbb{T} \rightarrow \mathbb{T}$. Then the diffeomorphism $g^{-1}: \partial \mathbb{T} \rightarrow \partial \mathbb{T}$ extends to the diffeomorphism $\widetilde{g}^{-1}: \mathbb{T} \rightarrow \mathbb{T}$. Note that $\widetilde{g}^{-1}$ is orientation-preserving. Let $C$ denote the cobordism from $-\partial \mathbb{T}$ to $S^3$ obtained by stacking the cobordism $(-\partial \mathbb{T} \times [0,\frac{1}{2}]) \mathrel{\cup_{g}} (-\partial \mathbb{T} \times [\frac{1}{2},1])$ from $-\partial \mathbb{T}$ to $-\partial \mathbb{T}$ on top of the punctured $\mathbb{T}$ (with puncture denoted by $*$). Let $C'$ denote the cobordism from $-\partial \mathbb{T}$ to $S^3$ obtained by stacking the identity cobordism  $(-\partial \mathbb{T} \times [0,\frac{1}{2}]) \mathrel{\cup_{id}} (-\partial \mathbb{T} \times [\frac{1}{2},1])$ from $-\partial \mathbb{T}$ to $-\partial \mathbb{T}$ on top of $\mathbb{T}$ punctured at $\widetilde{g}^{-1}(*)$. The orientation-preserving diffeomorphism $\widetilde{g}^{-1}: \mathbb{T} \rightarrow \mathbb{T}$ gives us the following orientation-preserving diffeomorphism $\Theta: C \rightarrow C'$: on $-\partial \mathbb{T} \times [0,\frac{1}{2}]$, we take $\Theta$ to be the identity; on $-\partial \mathbb{T} \times [\frac{1}{2},1]$, we take $\Theta$ to be $g^{-1} \times id$; and on the punctured $\mathbb{T}$, we take $\Theta$ to be $\widetilde{g}^{-1}$. Let $\mathfrak{s}_C$ denote the unique Spin$^c$ structure on $C$ and let $\mathfrak{s}_{C'}$ denote the unique Spin$^c$ structure on $C'$. We then get this commutative diagram:

\begin{center}
$\begin{CD}
HF^{+}(-\partial \mathbb{T}, \mathfrak{s}) @>F^{+}_{C, \mathfrak{s}_C}>> HF^{+}(S^3) \\
@VV(\Theta|_{-\partial \mathbb{T}})_{\ast} V @VV(\Theta|_{S^3})_{\ast} V\\
HF^{+}(-\partial T, \mathfrak{s}) @>F^{+}_{C', \mathfrak{s}_{C'}}>> HF^{+}(S^3).
\end{CD}$
\end{center}

\noindent Note that
\[
F^{+}_{C, \mathfrak{s}_C} = F^{+}_{\mathbb{T}, \mathfrak{s}} \circ g^{*},  
\]
\[
F^{+}_{C', \mathfrak{s}_{C'}} = F^{+}_{\mathbb{T}, \mathfrak{s}},
\]
\[
(\Theta|_{-\partial \mathbb{T}})_{\ast} =id.
\]
From the lemma, $F^{+}_{C, \mathfrak{s}_C} (c^{+}(\xi)) =0$ and $F^{+}_{C', \mathfrak{s}_{C'}} (c^{+}(\xi)) \neq 0$, but this contradicts the commutativity of the diagram. This concludes the proof of Theorem \ref{Thm1}.

\begin{Remark}
The above argument can also be used to show that the smooth 4-manifold $\mathbb{T}'$ obtained by starting with $B^4 \mathrel{\natural} \big( G \times \mathbb{W} \big)$ and twisting a copy of $\mathbb{W}$ in $1 \times \mathbb{W}$ is a $G$-cork. However it is not yet know if $\mathbb{T}'$ admits an embedding into a closed smooth 4-manifold so that removing and regluing using the $|G|$ twists produces $|G|$ distinct smooth structures.
\end{Remark}

We now prove Theorem \ref{Thm2}. Define the action of $G$ on $HF^{+}(-\partial \mathbb{T}, \mathfrak{s})$ by: $g \cdot x = g^{*}(x)$. To see that this is effective, we need to show that for any $g\neq 1$ there is an $x \in HF^{+}(-\partial \mathbb{T}, \mathfrak{s})$ so that $g^{*}(x) \neq x$. So fix $g\neq 1$. The above lemma implies that $g^{*}(c^{+}(\xi)) \neq c^{+}(\xi)$. Hence we can take our $x$ to be $c^{+}(\xi)$. This concludes the proof of Theorem \ref{Thm2}.

\end{document}